\documentclass[11pt]{amsart}

\theoremstyle{plain}
\newtheorem{thm}{Theorem}[section]
\newtheorem{theorem}[thm]{Theorem}

\newtheorem{lemma}[thm]{Lemma}

\newtheorem{proposition}[thm]{Proposition}
\theoremstyle{definition}
\newtheorem{remark}[thm]{Remark}

\newtheorem{definition}[thm]{Definition}

\newtheorem{conjecture}[thm]{Conjecture}

\numberwithin{equation}{section}

\newcommand{\sC}{{\mathcal C}}
\newcommand{\sD}{{\mathcal D}}

\newcommand{\sH}{{\mathcal H}}

\newcommand{\sK}{{\mathcal K}}
\newcommand{\sL}{{\mathcal L}}

\newcommand{\sP}{{\mathcal P}}

\newcommand{\C}{{\mathbb C}}

\newcommand{\BP}{{\mathbb P}}

\newcommand{\Z}{{\mathbb Z}}

\newcommand{\End}{{\rm End}}

\newcommand{\fg}{{\mathfrak g}}

\newcommand{\fgl}{{\mathfrak g}{\mathfrak l}}

\newcommand{\fcsp}{{\mathfrak c}{\mathfrak s}{\mathfrak p}}

\newcommand\tr{{\rm tr}}

\newcommand\Fr{\mathop{\rm Frame}\nolimits}

\def\Sym{\mathop{\rm Sym}\nolimits}

\def\Hom{\mathop{\rm Hom}\nolimits}

\title[Legendrian cone structures and contact prolongations]{Legendrian cone structures and contact prolongations}
\author[Jun-Muk Hwang]{Jun-Muk Hwang} 
\address{Institute for Basic Science, Center for Complex Geometry,
55 Expo-ro, Yuseong-gu, Daejeon, 34126, Korea}
 \email{jmhwang@ibs.re.kr}

\begin{document}

\begin{abstract}
We study a cone structure $\sC \subset \BP D$ on a holomorphic contact manifold $(M, D \subset T_M)$ such that each fiber $\sC_x \subset \BP D_x$ is isomorphic to a Legendrian submanifold of fixed isomorphism type.
By characterizing subadjoint varieties among Legendrian submanifold in terms of contact prolongations,   we prove that the canonical distribution on the associated contact G-structure admits a holomorphic horizontal splitting.
\end{abstract}

\maketitle



%

\section{Introduction}
We will work in the complex analytic setting: all geometric objects refer to holomorphic ones.

A contact structure on  a complex manifold $M$ of dimension $2n+1$ is a
subbundle $D \subset T_M$ of rank $2n$ such that the Lie bracket of vector fields induces a nondegenerate bilinear form
$$ \omega_x : \wedge^2 D_x \to L_x$$ for each $x \in M$,  where
 $L$ is the quotient line bundle  $ T_M/D$. The following conjecture is well-known.

 \begin{conjecture}\label{c.LS} Let $X$ be a Fano manifold with a contact structure $D \subset T_X$. Assume that ${\rm Pic} X \cong \Z \cdot L$ for the line bundle $L = T_X /D$. Then $X$ is homogeneous. \end{conjecture}

  Conjecture \ref{c.LS} has drawn much attention because it implies the LeBrun-Salamon conjecture on quaternion-K\"ahler manifolds in Riemannian holonomy theory, for which we may assume additionally that $X$ is K\"ahler-Einstein (see \cite{Le}).

 One approach to Conjecture \ref{c.LS} is via contact lines and their varieties of minimal rational tangents. To explain this approach, we recall the following definitions.

\begin{definition}\label{d.Legendre}
Let $(V, \sigma)$ be a symplectic vector space, i.e., a vector space $V$ equipped with a symplectic form $\sigma: \wedge^2 V \to \C$.
\begin{itemize}
\item[(1)] The {\em conformal symplectic group} ${\rm CSp}(V) \subset {\rm GL}(V)$ is the direct sum of the symplectic group ${\rm Sp}(V)$ and the center $\C^{\times} \cdot {\rm Id}_V$.
    \item[(2)] A projective subvariety $Z \subset \BP V$ is {\em Legendrian} if its affine cone $\widehat{Z} \subset V$ satisfies
$ \dim \widehat{Z} = \frac{1}{2} \ \dim V $ and is isotropic with respect to $\sigma$, i.e.,  $$\omega(T_{\widehat{Z}, v}, T_{\widehat{Z}, v}) =0 $$ for each nonsingular point
$v \in \widehat{Z}$.  Here, we regard the tangent space $T_{\widehat{Z},z}$ as a subspace of $V$.
We say that $Z$ is a Legendrian submanifold if furthermore it is nonsingular.
 \item[(3)]  A Legendrian submanifold $Z \subset \BP V$ is {\em nondegenerate} if the affine cone $\widehat{Z} \subset V$ spans the vector space $V$.
\item[(4)] A nondegenerate Legendrian submanifold $Z \subset \BP V$ is a {\em subadjoint variety  } if the group $G^Z \subset {\rm CSp}(V)$ of elements of ${\rm CSp}(V)$ preserving $\widehat{Z} \subset V$
 acts transitively on $Z$. \end{itemize}
\end{definition}

The above definition of subadjoint varieties agrees with the standard one by  Theorem 11 of \cite{LM}.
Many examples of Legendrian submanifolds in $\BP V$ other than subadjoint varieties are given  in  \cite{Bu07}, \cite{Bu08} and \cite{LM}.

In the setting of Conjecture \ref{c.LS}, a rational curve $C \subset X$ is called a contact line if $C \cdot L =1$ and $C$ is tangent to $D \subset T_X$. It is known that contact lines exist through every point of $X$.  Kebekus proved the following result in Theorem 4.1 and Theorem 4.4 of \cite{Ke2}.

\begin{theorem}\label{t.Ke}
 For $X, D$ and $L$ as in Conjecture \ref{c.LS}, fix an irreducible component $\sK$ of the space of contact lines such that members of $\sK$ cover $X$. Let   $x \in X$ be a general point and $\sK_x \subset \sK$ be the subscheme parametrizing members of $\sK$ passing through $x$. Then \begin{itemize}
 \item[(a)] all members of $\sK_x$ are nonsingular;
 \item[(b)] the algebraic subset $\sC_x \subset \BP D_x$ consisting of tangent spaces to members of $\sK_x$ is nonsingular;
     and
       \item[(c)] each irreducible component of $\sC_x$ is a Legendrian submanifold for the symplectic form $\omega_x$ on $D_x$. \end{itemize} \end{theorem}

     The subset $\sC_x \subset \BP D_x$ in Theorem \ref{t.Ke} is precisely the variety of minimal rational tangents of the Fano manifold $X$ in the sense of \cite{HM98} or \cite{H01}. By applying the theory of varieties of minimal rational tangents, Mok obtained the following result (Main Theorem in Section 2 of \cite{Mo}).

     \begin{theorem}\label{t.Mok}
     In the setting of Theorem \ref{t.Ke}, if $\sC_x$ is a subadjoint variety, then $X$ is homogeneous. \end{theorem}

     This seems to be a promising approach to Conjecture \ref{c.LS}, as it does not assume the existence of  vector fields on $X$.
 However, no good strategy has been suggested so far  to prove $\sC_x$ is a subadjoint variety in Theorem \ref{t.Ke}. This gives a motivation to study the geometry of general `Legendrian cone structure' in the following sense.

\begin{definition}\label{d.cone-structure}
Let $D \subset T_M$ be a contact structure on a complex manifold $M$ and let $\omega_x : \wedge^2 D_x \to L_x$ be the associated symplectic form at each $x \in M$. A submanifold $\sC \subset \BP D$ is a {\em Legendrian cone structure} on $M$ if each fiber $\sC_x \subset \BP D_x$ is a  Legendrian submanifold in the sense of Definition \ref{d.Legendre}. A Legendrian cone structure is {\em nondegenerate} if $\sC_x$ is nondegenerate. For a  fixed Legendrian submanifold $Z \subset \BP V$  in the sense of Definition \ref{d.Legendre}, a Legendrian cone structure $\sC \subset \BP D$ on $M$ is $Z$-{\em isotrivial} if the submanifold $\sC_x \subset \BP D_x$ is equivalent to $Z \subset \BP V$  modulo ${\rm CSp}(V)$ for each $x \in M$. \end{definition}

 A theory of general Legendrian cone structures is yet to be developed. It seems reasonable to start with studying $Z$-isotrivial Legendrian cone structures for a general nondegenerate $Z \subset \BP V$.
When $Z \subset \BP V$ is a subadjoint variety, a $Z$-isotrivial Legendrian cone structure is a parabolic contact structure in the sense of Section 4.2 of \cite{CS}. In fact, the proof of Theorem \label{t.Mok} uses the result of \cite{Ho} which employed  Tanaka connections  for parabolic contact structures.
As a first step toward the geometry of $Z$-isotrial Legendrian cone structures, one may ask whether there is an analog of Tanaka connection in this general setting.   We prove  the following Theorem, which says that the canonical distribution on the associated contact G-structure admits a holomorphic horizontal splitting. This can be viewed as an analog of Tanaka connection.

\begin{theorem}\label{t.main}
Let $Z \subset \BP V$ be a nondegenerate Legendrian submanifold for a symplectic vector space $(V, \sigma)$, which is not a subadjoint variety and  let $G\subset {\rm CSp}(V)$ the automorphism group of $Z$.
Let $M$ be a complex manifold  $\dim M = \dim V +1$  equipped with a contact structure $D \subset T_M$ and a $Z$-isotrival Legendrian cone structure $\sC \subset \BP D$.  We have the associated  $G$-principal fiber subbundle $\pi:\sP \to M$ of the contact frame bundle of $M$. Let $T^{\pi} \subset T_{\sP}$ be the kernel of ${\rm d} \pi$ and let $\sD \subset T_{\sP}$ be the natural distribution satisfying $\sD/T^{\pi} \cong \pi^* D$. Then there exists a  subbundle $\sH \subset \sD$ such that $\sD \cong T^{\pi} \oplus \sH$. \end{theorem}

Theorem \ref{t.main} is an immediate consequence of two results, Theorem \ref{t.csplit} and Theorem \ref{t.1} below. Theorem \ref{t.csplit}  is a general result on  contact G-structures and contact prolongations. In principle, Theorem \ref{t.csplit} is contained in Morimoto's general theory (\cite {Mr}) of G-structures on filtered manifolds. As it can be stated and proved more explicitly in the contact setting, we present a self-contained account here. Theorem \ref{t.1} is a characterization of
subadjoint varieties in terms of  contact prolongations, a contact analog of Theorem 7.13 in \cite{FH18}. Its proof is simpler than that of \cite{FH18}, thanks to the work \cite{LM}.
  These  theorems are expected to be useful in the approach to
 Conjecture \ref{c.LS} via varieties of minimal rational tangents of contact lines.

\section{Contact prolongation and contact G-structure}\label{s.2}

In this section, we look at the theory of contact G-structures.
We will give a self-contained presentation, modifying the ordinary G-structure theory  presented in \cite{St}. Although this is a special case of
the general theory of  geometric structures on
filtered manifolds developed by Morimoto in \cite{Mo}, the explicit presentation here could be useful
in applications.

\begin{definition}\label{d.prolong}
Let $(V, \sigma)$ be a symplectic  vector space.  \begin{itemize}
\item[(1)] The {\em conformal symplectic algebra} is the Lie algebra
$\fcsp(V) \subset \fgl(V)$ of the conformal symplectic group ${\rm CSp}(V) \subset {\rm GL}(V)$ in Definition \ref{d.Legendre}. Recall that  $ a \in \fgl(V)$ belongs to $\fcsp(V)$ if and only if $$ \sigma(a(v),w) + \sigma(v, a(w)) =
\frac{2}{\dim V}\tr (a) \cdot \sigma(v, w) $$  for all $ v,w \in V
,$ where $\tr (a)$ is the trace of $a \in \End(V)$.
\item[(2)] For a Lie subalgebra $\fg \subset \fcsp(V)$ and an element $A \in \Hom(V, \fg)$,
  we will write $A(u) \in \fg, u \in V,$ simply as $A_u \in \End(V)$.  Denote by $\vec{A} \in V$ the unique vector satisfying
  $$\sigma( \vec{A}, u) = \frac{2}{\dim V} \tr (A_u) \mbox{ for all } u \in V.$$
  \item[(3)]  For $A \in \Hom(V, \fg)$, let $\delta A $ be the element of $\Hom(\wedge^2 V, V)$ defined by
  $$\delta A (u, v) = A_u(v) - A_v(u)  - \sigma (u, v) \ \vec{A} $$  for $ u, v \in V.$
    This defines  a homomorphism $\delta: \Hom( V, \fg) \to \Hom (\wedge^2 V, V).$
  An element $A \in \Hom(V, \fg)$ is a {\em contact prolongation} of $\fg$ if $ \delta A =0$.
  \end{itemize}
\end{definition}

\begin{definition}\label{d.frame}
Let $M$ be a complex manifold with a contact structure $D \subset M$ and the contact line bundle $L= T_M/D$.
Denote by $\omega: \wedge^2 D \to L$ the homomorphism induced by the Lie bracket of local vector fields, equipping $D_x$ with  the  $L$-valued symplectic form  $\omega_x:  \wedge^2 D_x \to L_x$. Fix a symplectic vector space $(V, \sigma)$ with
$\dim M = \dim V +1$.  \begin{itemize}
 \item[(1)] A {\em contact frame} at $x \in M$ is  a linear isomorphism $f: V \to D_x$  such that
 $\omega_x(f(u), f(v)) =0$ for any $u, v \in V$ satisfying $\sigma(u, v) =0$. There exists  a unique nonzero vector $\vec{f} \in L_x $ such that $$\omega_x( f(u), f(v)) = \sigma( u, v) \ \vec{f} $$  for all $ u,v \in V.$
     \item[(2)]
Denote by $\Fr_x(M,D)$  the set of all contact frames at $x$. The union $$\Fr(M,D) := \cup_{x \in M} \Fr_x(M,D)$$ is a principal ${\rm CSp}(V)$-bundle over $M$, called the {\em contact frame bundle} of $(M, D)$. The right action  $$R_g: \Fr(M,D) \to \Fr(M,D)$$ of
$g \in {\rm CSp}(V)$ is given by $$ f \in \Fr(M,D) \ \mapsto R_g (f) = f \circ g \in \Fr(M,D).$$ Note that the subgroup ${\rm Sp}(V) \subset {\rm CSp}(V)$ acts trivially on $\vec{f}$, while
an element $c \cdot {\rm Id}_V, c \in \C^{\times},$ sends $\vec{f} $ to $c^{2} \vec{f}.$
\item[(3)] For a closed subgroup $G \subset {\rm CSp}(V)$,
a $G$-principal subbundle $\sP \subset \Fr(M,D)$ is called a {\em contact} $G$-{\em structure} on $M$. Denote by $\pi: \sP \to M$ the natural projection.
\item[(4)] In (3),  denote by $\lambda$ the unique 1-form on $\sP$ whose value $\lambda(\vec{w}) \in \C$ at $\vec{w} \in T_{\sP, f}$ satisfies $$ \lambda(\vec{w}) \vec{f} = {\rm d} \pi (\vec{w}) \mod D_x, \ x = \pi(f).$$
  Let $\sD \subset T_{\sP}$ be the  subbundle of corank 1 annihilated by $\lambda$.  The  bundle $T^{\pi} := {\rm Ker}( {\rm d} \pi)$ is contained in $\sD$ and there is a natural isomorphism  of vector bundles $\sD / T^{\pi} \cong \pi^* D$.
\end{itemize}
\end{definition}

The following is our main result on contact $G$-structures.

\begin{theorem}\label{t.csplit}
Let $G \subset {\rm CSp}(V)$ be a closed  subgroup such that its Lie algebra $\fg \subset \fcsp(V)$ has no nonzero contact prolongation. Let $\sP \subset \Fr(M,D)$ be a contact $G$-structure on a contact manifold $(M,D)$ with $\dim M = \dim V +1$ and let $T^{\pi} \subset \sD \subset T_{\sP}$ be as in Definition \ref{d.frame} (4). Then there exists a
subbundle $\sH \subset \sD$ inducing a splitting $\sD \cong T^{\pi} \oplus \sH$.
\end{theorem}

\begin{remark} The subbundle $\sH$ in Theorem \ref{t.csplit} is not necessarily $G$-equivariant. So it is not a (partial) connection on the principal bundle $\sP$. \end{remark}

The rest of the section is devoted to the proof of Theorem \ref{t.csplit}.
We start with some definitions.

\begin{definition}\label{d.csolder}
For a contact $G$-structure  $\sP \subset \Fr(M,D),$ we have the bundles $T^{\pi} \subset \sD \subset T_{\sP}$ from
Definition \ref{d.frame}. \begin{itemize}
\item[(1)]  For each $f \in \sP$ and $x = \pi(f) \in M$, let $\theta_f: \sD_{f} \to V$ be the
homomorphism obtained by the composition
$$ \sD_f \stackrel{{\rm d} \pi}{\longrightarrow} D_x \stackrel{f^{-1}}{\longrightarrow} V$$ such that ${\rm Ker}(\theta_f) = T^{\pi}_f$. Let $\theta: \sD \to \sP \times V$ be the homomorphism of vector bundles on $\sP$ given by $\{ \theta_f, \ f \in \sP\}.$.
\item[(2)]
Suppose there exist a line subbundle $\sL \subset T_M$ such that $ D \cap \sL=0$.
Define a $V$-valued 1-form $\theta^{\sL}$ on $\sP$ by
\begin{eqnarray*} \theta^{\sL} (\vec{w}) & = & \theta(\vec{w}) \mbox{ if } \vec{w} \in \sD \mbox{ and } \\
\theta^{\sL} (\vec{w}) &=& 0 \mbox{ if } {\rm d} \pi (\vec{w}) \in \sL.
\end{eqnarray*} \end{itemize}
\end{definition}

\begin{lemma}\label{l.Sternberg}
In the notation of Definitions \ref{d.frame} and  \ref{d.csolder},   denote by $\widetilde{a}$ the fundamental vector field on $\sP$ generated by $a \in \fg$.   Then  for any
    $\vec{w} \in T_{\sP}$, we have
   $$ {\rm d } \lambda (\widetilde{a}, \vec{w}) = - \frac{2}{\dim V} \ \tr(a) \ \lambda (\vec{w}) \mbox{  and  }
{\rm d} \theta^{\sL} ( \widetilde{a}, \vec{w} ) = - a \cdot \theta^{\sL}( \vec{w} ).$$
 \end{lemma}

\begin{proof}  It suffices to prove it when $G= {\rm CSp}(V),$ because  the forms $\lambda$ and $\theta^{\sL}$ for a subbundle $\sP \subset \Fr(M,D)$ are just restrictions
of the corresponding forms  on $\Fr(M,D)$.

We claim that
\begin{itemize}
\item[(i)] $R_g^* \lambda = c^{-2} \lambda$ for $g = (g',c \ {\rm Id}_V) \in  {\rm Sp}(V) \times \C^{\times} \cdot {\rm Id}_V$; and
    \item[(ii)]
  $R_g^* \theta^{\sL} = g^{-1} \cdot \theta^{\sL}$  for any $g \in {\rm CSp}(V).$ \end{itemize}
To prove (i),  note that  ${\rm d} \pi (\vec{w}) = {\rm d} \pi (R_{g*} \vec{w})$ for $\vec{w} \in T_{\sP, f}$.
From Definition \ref{d.frame}, we obtain, modulo $D_x, x = \pi(f),$
$$\lambda( R_{g*} \vec{w}) R_g(\vec{f}) \equiv  {\rm d} \pi (R_{g*} \vec{w})  \equiv {\rm d} \pi (\vec{w}) \equiv \lambda(\vec{w}) \vec{f}.$$ Since $R_g(\vec{f}) = c^2 \ \vec{f}$, we obtain (i).

To prove (ii),
if $\vec{w} \in \sD_f$,
$$ \theta^{\sL}(R_{g*} \vec{w}) =  (f \circ g)^{-1} ({\rm d} \pi (\vec{w})) = g^{-1} \circ f^{-1} ({\rm d } \pi (\vec{w})) = g^{-1} \cdot \theta^{\sL}(\vec{w}).$$
On the other hand, if ${\rm d} \pi (\vec{w}) \in \sL$, then ${\rm d} \pi (R_{g_*}\vec{w}) \in \sL$. Thus
$$\theta^{\sL}(R_{g*} \vec{w}) = 0 = \theta^{\sL}(\vec{w}).$$ This proves (ii).

The derivative of (i) shows
$${\rm Lie}_{\widetilde{a}} \lambda = - \frac{2}{\dim V} \tr (a)  \  \lambda$$
where ${\rm Lie}_{\widetilde{a}}$ denotes the Lie derivative with respect to the vector field $\widetilde{a}$.
Since $\lambda(\widetilde{a}) = 0$, Cartan's formula for the Lie derivative of forms (e.g, Equation (1.9) in Chapter 3 of \cite{St}) gives
${\rm d} \lambda (\widetilde{a}, \vec{w}) = -\frac{2}{ \dim V} \ \tr(a)  \ \lambda(\vec{w}).$
Similarly, the derivative of (ii) shows
$${\rm Lie}_{\widetilde{a}} \theta^{\sL}= - a \cdot \theta^{\sL}.$$ Since $\theta^{\sL}( \widetilde{a}) = 0$, Cartan's formula gives ${\rm d} \theta^{\sL} ( \widetilde{a}, \vec{w}) = - a \cdot \theta^{\sL}( \vec{w} ).$
\end{proof}

\begin{lemma}\label{l.ii}
In the notation of Definition \ref{d.frame} (4), let $H \subset \sD_f$ be a subspace such that
$\sD_f = T^{\pi}_f \oplus H.$
For $u, v \in V$, let $u^H, v^H \in H$ be the corresponding vectors by the two isomorphisms $$V \stackrel{f}{\longrightarrow} D_x \stackrel{{\rm d}\pi}{\longleftarrow} H.$$ Then
${\rm d} \lambda (u^H, v^H) = - \sigma( u, v).$
\end{lemma}

\begin{proof}
Choose local sections $\widetilde{u}$ and $\widetilde{v}$ of $\sD$ near $f \in \sP$ whose values at $f$ are
$u^H$ and $v^H$, respectively. Then \begin{eqnarray*}
 {\rm d} \lambda (\widetilde{u}, \widetilde{v})
&=& \widetilde{u}(\lambda(\widetilde{v})) - \widetilde{v}(\lambda(\widetilde{u}))  - \lambda ([\widetilde{u}, \widetilde{v}]) \\ & = &  - \lambda ([\widetilde{u}, \widetilde{v}]). \end{eqnarray*}
Thus \begin{eqnarray*} {\rm d} \lambda (u^H, v^H)  \ \vec{f}  &=& {\rm d} \lambda (\widetilde{u}, \widetilde{v})_f \ \vec{f} \\
&=&  - \lambda([\widetilde{u}, \widetilde{v}])_f \ \vec{f} \\ &=& - {\rm d} \pi([\widetilde{u}, \widetilde{v}])_x \   \mod D_x, \ x = \pi (f) \\
&=&  - \omega_x ({\rm d} \pi (u^H), {\rm d} \pi (v^H)) \\
&=& - \sigma( u, v) \ \vec{f}. \end{eqnarray*}
\end{proof}

\begin{lemma}\label{l.null}
As in Lemma \ref{l.ii}, let $H \subset \sD_f$ be a subspace such that
$\sD_f = T^{\pi}_f \oplus H.$ Define
$${\rm Null}^{{\rm d} \lambda}_H :=\{ \vec{w} \in T_{\sP, f}, \ {\rm d} \lambda (\vec{w}, \vec{v}) =0 \mbox{ for all } \vec{v} \in H\}.$$ Then there exists a unique 1-dimensional subspace $\ell_H \subset T_{M,x}, x = \pi(f)$, such that
$$ ({\rm d}_{f} \pi)^{-1} (\ell_H) = {\rm Null}^{{\rm d} \lambda}_H \ \mbox{ and } \ {\rm Null}^{{\rm d} \lambda}_H /T^{\pi}_f  \cong  \ell_H.$$ \end{lemma}

\begin{proof}  Lemma \ref{l.Sternberg}  implies $T^{\pi}_{f} \subset {\rm Null}^{{\rm d} \lambda}_H.$
 Lemma \ref{l.ii} shows that   ${\rm d} \lambda |_H$ is nondegenerate. Thus $$ H \cap {\rm Null}^{{\rm d} \lambda}_H =0.$$ Choose a subspace
$H^+ \subset T_{\sP, f} $ such that $$H \subset H^+ \mbox{ and } T_{\sP, f} = T^{\pi}_{f} \oplus H^+.$$ Then ${\rm d} \lambda |_{H^+}$ cannot be nondegenerate because $\dim H^+ = \dim D_x +1$ is odd.
It follows that $$ {\rm Null}_H^{{\rm d} \lambda} = T^{\pi}_f \oplus  (H^+ \cap {\rm Null}^{{\rm d} \lambda}_H)
\mbox{ and } \dim H^+ \cap {\rm Null}^{{\rm d} \lambda}_H = 1.$$ Thus $\ell_H := {\rm d} \pi (H^+ \cap {\rm Null}^{{\rm d} \lambda}_H)$ satisfies the requirement. It is clear that $\ell_H$ is uniquely determined by $H$. \end{proof}

\begin{lemma}\label{l.Pi} In the setting of Lemma \ref{l.null},
define $\Pi_H \in \Hom(\wedge^2 V, V)$ as follows.
Extend $\ell_H \subset T_{M,x}$ in Lemma \ref{l.null} to a line subbundle $\sL_H \subset T_M$ complementary to $D$ in a neighborhood of $x$. By Definition \ref{d.csolder} (2), we have the $V$-valued 1-form $\theta^{\sL_H}$ in a neighborhood of  $\pi^{-1}(x) \subset \sP.$
For $u, v \in V$, let $u^H, v^H \in H$ be as in Lemma \ref{l.ii} and define
$$\Pi_H(u, v) := {\rm d} \theta^{\sL_H} ( u^H, v^H).$$ Then \begin{itemize} \item[(i)] $\Pi_H$ is determined by $H$, independent of the choice of $\sL_H$;  and \item[(ii)]   $ {\rm d} \lambda (v^H, \vec{w}) = \sigma ( \theta^{\sL_H}(\vec{w}), v) $ for any $ v \in V$ and $ \vec{w} \in T_{\sP, f}.$ \end{itemize} \end{lemma}

\begin{proof}
To check (i), let $\sL'_H \subset T_M$ be another line subbundle in a neighborhood of $x$ complementary to $D$ such that
$$(\sL'_H)_x = (\sL_H)_x = \ell_H.$$
Then $\beta := \theta^{\sL'_H} - \theta^{\sL_H}$ is a $V$-valued 1-form in a neighborhood of $\pi^{-1}(x)$ which vanishes on $\sD$. Thus we can write $\beta = h \cdot \lambda$ for some $V$-valued function $h$. We have $h|_{\pi^{-1}(x)} = 0$ from $(\sL'_H)_x = (\sL_H)_x.$ Then $$ {\rm d} \theta^{\sL'_H} - {\rm d} \theta^{\sL_H}= {\rm d} \beta = {\rm d} h \wedge \lambda + h \ {\rm d} \lambda$$ satisfies
${\rm d} \beta (\vec{u}, \vec{v}) = 0$ for any $\vec{u}, \vec{v} \in \sD_f,$ because $\lambda (\vec{u}) = \lambda (\vec{v}) =0$ and $h(f) =0$. Thus $\Pi_H$ does not depend on the choice of $\sL_H$.

To check (ii), first assume that $\vec{w} \in \sD_f$.  Then $\theta^{\sL_H}(\vec{w}) = f^{-1} ({\rm d} \pi (\vec{w})).$ As in the proof of Lemma \ref{l.ii}, choose  local sections $\widetilde{v}$ and $\widetilde{w}$ of $\sD$ near $f \in \sP$   whose values at $f$ are $v^H$ and $\vec{w}$, respectively. Then the same argument as in the proof of Lemma \ref{l.ii}  shows
\begin{eqnarray*}   {\rm d} \lambda (v^H, \vec{w}) \ \vec{f}
& = & - \omega_x( {\rm d} \pi (v^H), {\rm d} \pi (\vec{w}))  \\ &=& - \sigma( v, \theta^{\sL_H}(\vec{w})) \ \vec{f}. \end{eqnarray*}
 This proves (ii) when $\vec{w} \in \sD_{(f, \vec{f})}.$
Since ${\rm Null}^{{\rm d} \lambda}_H \cap \sD_f$ has codimension 1 in ${\rm Null}^{{\rm d} \lambda}_H,$
it remains to check  (ii)  when $\vec{w} \in {\rm Null}^{{\rm d}\lambda}_H$. But in this case,
the left hand side of (ii) vanishes by the definition of ${\rm Null}^{{\rm d} \lambda}_H$, while the right hand side of (ii) vanishes by the definitions of $\ell_H$ and  $\theta^{\sL_H}$.
This completes the proof of (ii).
\end{proof}

\begin{definition}\label{d.cs}
In the setting of Lemma \ref{l.Pi}, let $H$ and $H'$ be two subspaces of $\sD_f$ complementary to $T^{\pi}_f$. Then we have  $s_{H,H'} \in \Hom(V, \fg)$  defined in the following way. For $u \in V$, we have vectors
$u^H \in H$ and $u^{H'} \in H'$ defined as in Lemma \ref{l.ii}. Then $u^{H} - u^{H'} \in T^{\pi}_{f}$. We define $s_{H,H'}(u) \in \fg$ as the unique element such that  $$ \widetilde{ s_{H, H'}(u)}_f = u^H - u^{H'},$$ namely,
the fundamental vector field on $\sP$ generated by $s_{H,H'}(u)$ has value $u^{H} - u^{H'}$ at $f$.
\end{definition}

The following is immediate.

\begin{lemma}\label{l.cs}
In Definition \ref{d.cs}, fix a choice of $H \subset \sD_{f}$ complementary to $T^{\pi}_{f}$. Then the map $H' \mapsto s_{H, H'} \in \Hom(V, \fg)$ gives an identification of the set of complementary subspaces of $T^{\pi}_{f} \subset \sD_{f}$ and $\Hom(V, \fg)$. \end{lemma}

\begin{lemma}\label{l.vec}
In Definition \ref{d.cs},
let $\vec{s}_{H,H'} \in V$ be the vector associated to $s_{H,H'} \in \Hom(V, \fg)$ as in Definition \ref{d.prolong} (2). Then   $$- \lambda \cdot \vec{s}_{H,H'} = \theta^{\sL_H} - \theta^{\sL_{H'}} $$ on $ T_{\sP,f}.$ Consequently,
$$- {\rm d} \lambda (\vec{u}, \vec{v}) \ \vec{s}_{H,H'}  = {\rm d} \theta^{\sL_{H}} (\vec{u}, \vec{v}) - {\rm d} \theta^{\sL_{H'}} (\vec{u}, \vec{v}) $$ for any $\vec{u}, \vec{v} \in \sD_f.$\end{lemma}

\begin{proof}
From Definition \ref{d.prolong} (2),
it suffices to check
$$\sigma(\theta^{\sL_{H}}(\vec{w}) -\theta^{\sL_{H'}}(\vec{w}), v) =   - \frac{2}{\dim V} \tr (s_{H,H'}(v)) \ \lambda( \vec{w})$$ for any  $v \in V$ and $\vec{w} \in T_{\sP,f}.$
But Lemma \ref{l.Pi} (ii) implies  $$\sigma(\theta^{\sL_{H}}(\vec{w}) -\theta^{\sL_{H'}}(\vec{w}), v ) = {\rm d} \lambda ( v^{H} - v^{H'}, \vec{w}) = {\rm d} \lambda (\widetilde{s_{H,H'}(v)}, \vec{w}),$$ which is equal to
$  - \frac{2}{\dim V} \tr ( s_{H,H'}(v)) \ \lambda (\vec{w})$ by Lemma \ref{l.Sternberg}. This completes the proof. \end{proof}

\begin{lemma}\label{l.cformula}
 In Definition \ref{d.cs},
    we have $$\Pi_{H'} - \Pi_{H} = \delta s_{H, H'}$$ where $\delta$ is as in Definition \ref{d.prolong} (3). \end{lemma}

\begin{proof}
For $u, v \in V$,
let
$\widetilde{s_u}$ (resp. $\widetilde{s_v}$) be the fundamental vector field on $\sP$
 corresponding to $s_u:= s_{H,H'}(u) \in \fg$ (resp. $s_v:= s_{H,H'}(v) \in \fg$). Then
 \begin{eqnarray*} \Pi_{H'}(u, v) - \Pi_H (u,v) & =& {\rm d}\theta^{\sL_{H'}}( u^{H'}, v^{H'}) - {\rm d} \theta^{\sL_H} (u^H, v^H) \\ & = &
 {\rm d} \theta^{\sL_{H'}} ( (u^{H'} - u^H), v^{H'}) + {\rm d} \theta^{\sL_{H'}}(u^H, v^{H'}) \\ & &  + {\rm d} \theta^{\sL_H} (u^H, (v^{H'} - v^H))  -{\rm d} \theta^{\sL_H}(u^H, v^{H'}) \\
 &=& -{\rm d} \theta^{\sL_{H'}} (\widetilde{s_u}, v^{H'}) - {\rm d} \theta^{\sL_H}(u^H, \widetilde{s_v}) \\ & & + {\rm d} \lambda (u^H, v^{H'}) \ \vec{s}_{H,H'} \\
 & = & s_{H, H'}(u) \cdot v - s_{H, H'}(v) \cdot u - \sigma( u, v) \ \vec{s}_{H,H'} \\
 & = & \delta s_{H, H'} (u, v), \end{eqnarray*}
where the third equality is from Lemma \ref{l.vec} and the fourth equality is from Lemma \ref{l.Sternberg} and
   Lemma \ref{l.ii}. This finishes the proof. \end{proof}

 \begin{proof}[Proof of Theorem \ref{t.csplit}]
 We use the terminology in Lemma \ref{l.cformula}.
Fix a subspace $W \subset \Hom(\wedge^2 V, V)$ complementary to ${\rm Im}(\delta)$ once and for all.
(This $W$ may not be stable under the natural $G$-action on $\Hom(\wedge^2 V, V)$.)

Let $f \in \sP$ be a given point. We claim that there exists a unique subspace $\sH_f
\subset \sD_f$ complementary to $T^{\pi}_f$ such that $\Pi_{\sH_f}
\in W$.

To see the existence, fix a complement $H \subset \sD_f$ to $T^{\pi}_f$. Then there exists an element $s \in \Hom(V, \fg)$ such that $ \Pi_H + \delta s \in W$. By Lemma \ref{l.cs}, we have $s=  s_{H, H'}$ for some $H' \subset \sD_f$ complementary to
$T^{\pi}_{f}$. Then Lemma \ref{l.cformula} gives
$ \Pi_{H'} = \Pi_H + \delta s_{H, H'}  \in W$. So we can put $\sH_{f} = H'$.

To see the uniqueness, suppose that $H, H' $ are two subspaces of $\sD_f$ complementary to $T^{\pi}_f$ satisfying $\Pi_H, \Pi_{H'} \in W$. Then
$\Pi_H - \Pi_{H'} \in W$, but $\Pi_H - \Pi_{H'} \in {\rm Im}(\delta)$ by Lemma \ref{l.cformula}. Thus we have $\Pi_H = \Pi_{H'}$ and $\delta s_{H,H'} =0$. By the assumption that $\fg$ has no nonzero contact prolongation, the homomorphism $\delta$ is  injective. Thus we have $s_{H, H'} =0$, which implies $H=H'$ by Lemma \ref{l.cs}.

By the claim, at each $f \in \sP$, we have a unique horizontal subspace
$\sH_{f} \subset \sD_f$. This defines a subbundle $\sH \subset \sD$ complementary to $T^{\pi}$. \end{proof}

\section{Contact prolongation and  Legendrian submanifolds}\label{s.1}

 Let $(V, \sigma)$ be a symplectic vector space.
 For a Legendrian submanifold $Z \subset \BP V,$ let  $G^Z \subset {\rm CSp}(V)$ be the subgroup in Definition \ref{d.Legendre}
 and let  $ \fg^Z \subset \fcsp(V)$ be its Lie algebra.
The main result of this section is the following.

\begin{theorem}\label{t.1}
A nondegenerate Legendrian submanifold $Z \subset \BP V$ is a subadjoint variety  if $\fg^Z \subset \fcsp(V)$ has a nonzero contact prolongation. \end{theorem}

The converse of Theorem \ref{t.1} holds, although we do not need it. One can see it  from the fact that
$Z$-isotrivial Legendrian cone structures are parabolic contact structures and the latter has prolongations (Section 4.2 of \cite{CS}).
To prove Theorem \ref{t.1}, it is convenient to use the following notion from \cite{Euler}.

\begin{definition}\label{d.Euler} Let $V$ be a vector space.
A subvariety $Z \subset \BP V$ is {\em Euler-symmetric} if for a general point $z \in Z$, there exists a $\C^{\times}$-subgroup
 $E_z \subset {\rm GL}(V)$ such that $E_z$ preserves $Z$ with an isolated fixed point at $z$ and acts trivially on $\BP T_{Z,z}$. \end{definition}

The following result is Proposition 2.7 in \cite{Euler}. We refer the reader to \cite{Euler} for the definition of fundamental forms of projective varieties.

\begin{proposition}\label{p.Euler}
Let $Z_1$ and $Z_2$ be two Euler-symmetric subvarieties of $\BP V$ of equal dimension. Let $z_1 \in Z_1$ and $z_2 \in Z_2$ be general points. If the systems of fundamental forms  of $Z_1$ at $z_1$ and   $Z_2$ at $z_2$ are isomorphic, then $Z_1$ and $Z_2$ are isomorphic by a projective transformation of $\BP V$. \end{proposition}

To apply this to Legendrian submanifolds, we recall some basic properties of fundamental forms of Legendrian submanifolds.
The following is proved in pp. 348-349 of \cite{LM}.
Recall that for a cubic form $P \in \Sym^3 W^{\vee}$ on a vector space $W$, its Hessian is
the homomorphism $\Sym^2 W \to W^{\vee}$ defined by $$w_1 \odot w_2 \in \Sym^2 W \ \mapsto \ P(w_1, w_2, \cdot) \in W^{\vee}.$$
We say that a cubic form $P$ is nondegenerate if its Hessian is surjective to $W^{\vee}$.

\begin{proposition}\label{p.III} Let $(V, \sigma)$ be a symplectic vector space.
for a point $z \in \BP V$, denote by $\widehat{z} \subset V$ the corresponding 1-dimensional subspace and by   $z^{\perp \sigma} \subset V$ the subspace of codimension 1 defined by $\sigma(\widehat{z}, \cdot) =0$.
Let $Z \subset \BP V$ be a Legendrian submanifold.
\begin{itemize} \item[(1)]
If $Z$ is not a linear subspace, then $Z$ is nondegenerate, i.e, the affine cone $\widehat{Z}$ spans $V$.
\item[(2)] If $Z$ is nondegenerate and $z \in Z$ is a general point, then the image of the second fundamental form
$${\rm II}_{Z,z}:  \Sym^2 T_{Z,z} \to N_{Z,z} = \Hom(\widehat{z}, V/T_{\widehat{Z}, z})$$ is
$\Hom(\widehat{z}, z^{\perp \sigma}/T_{\widehat{Z},z})$.
\item[(3)] In (2), the third fundamental form ${\rm III}_{Z,z}$ is given by a nondegenerate cubic form on $T_{Z,z}$
and
$${\rm II}_{Z,z}:  \Sym^2 T_{Z,z} \to \Hom(\widehat{z}, z^{\perp \sigma}/T_{\widehat{Z},z}) \cong (\widehat{z}^{\vee} )^{\otimes 2} \otimes T^{\vee}_{Z,z}$$ is isomorphic to the Hessian of this cubic form. In particular, the second fundamental form is determined by the third fundamental form. \end{itemize} \end{proposition}

The following is Theorem 2 of \cite{Bu07}. As the proof in \cite{Bu07} is somewhat complicated,
we will give a simple proof using Proposition \ref{p.III}.

\begin{proposition}\label{p.CSp} Let $(V, \sigma)$ be a symplectic vector space.
Let $Z \subset \BP V$ be a nondegenerate Legendrian submanifold. Then
an element $A \in {\rm GL}(V)$ preserving $\widehat{Z} \subset V$ belongs to ${\rm CSp}(V)$. \end{proposition}

\begin{proof}
To prove the proposition, it suffices to show that if $\sigma'$ is a symplectic form on $V$ with respect to which $Z$ is Legendrian,
then $\sigma' \in \C \cdot \sigma$.

Suppose that $\sigma' \not\in \C \sigma.$ Then there exists a linear combination
$\tau := \sigma' + c \cdot \sigma$ for some $c \in \C$ such that its null space
$${\rm Null}_{\tau} := \{ v \in V, \ \tau(v, u) = 0 \mbox{ for all } u \in V\}$$ is nonzero.
By Proposition \ref{p.III} (2), for a general $z \in Z$, we have
$z^{\perp \sigma} = z^{\perp \sigma'}. $
Thus ${\rm Null}_{\tau} \subset z^{\perp \sigma} = z^{\perp \sigma'}$ for all $z \in Z$.
This implies that the isomorphism $\iota_{\sigma}: \BP V \to \BP V^{\vee}$ given by $z \mapsto \BP z^{\perp \sigma}$
sends $Z$ into the linear subspace annihilating ${\rm Null}_{\tau}$.  This is a contradiction to the assumption that
 $Z$ is nondegenerate.  \end{proof}

Proposition \ref{p.CSp} says that when we apply results like Proposition \ref{p.Euler} to nondegenerate Legendrian submanifolds, we do not need to check whether  isomorphisms as projective subvarieties  respect the underlying symplectic form on the vector space: they always do.

\begin{proposition}\label{p.Z_P}
For a symplectic vector space $(V, \sigma),$  if a nondegenerate Legendrian submanifold $Z \subset \BP V$ is Euler-symmetric,
then it is a subadjoint variety. \end{proposition}

\begin{proof}
When $\dim V = 2n+2$,
for each  nondegenerate cubic polynomial $P$ in $n$ variables,
a special Legendrian subvariety $\widetilde{Z}_P \subset \BP^{2n+1}$ whose third fundamental form at a general point is isomorphic to $P$ is constructed in p. 357 of \cite{LM}.  Their construction implies that $\widetilde{Z}_P$ is Euler-symmetric. (In fact, this is a special case of a more general construction of Euler symmetric varieties in Definition 3.6  of \cite{Euler}, as mentioned in Example 3.10 of \cite{Euler}.)
By Proposition \ref{p.III} (3), the system of fundamental forms of a nondegenerate Legendrian submanifold is determined by its third fundamental form. Thus by Proposition \ref{p.Euler}, an Euler-symmetric Legendrian submanifold $Z \subset \BP V$ is projectively isomorphic to $\widetilde{Z}_P \subset \BP^{2n+1}$. Then by Corollary 26 of \cite{LM}, it is a subadjoint variety. \end{proof}

\begin{proposition}\label{p.formula}
Let $Z \subset \BP V$ be a Legendrian submanifold and let $\fg^Z \subset \fcsp(V)$ be as before.
Let
$A$ be a contact prolongation of  $\fg^Z$ with the corresponding vector $\vec{A} \in V$ as in Definition \ref{d.prolong} (2). Then for  a nonzero point $ v \in
\widehat{Z}$ and $w \in T_{\widehat{Z}, v}$, \begin{itemize}
\item[(i)] $A_{v}(v) = 2\sigma(\vec{A}, v) v$;
and \item[(ii)] $ A_{v}(w)  = \sigma(\vec{A}, v) w +
\sigma(\vec{A}, w) v.$ \end{itemize} \end{proposition}

\begin{proof} By the description of
$\fcsp(V)$ in Definition \ref{d.prolong}, we have $$\sigma(A_{v}(u), y) + \sigma(u,
A_{v}(y)) = \frac{2}{\dim V} \tr(A_v) \ \sigma(u,y)$$ for any $u, v, y \in V$.  Putting $v=y$
and using $\frac{2}{\dim V}\tr(A_v)   = \sigma (\vec{A}, v)$, we obtain
$$\sigma(A_{v}(u), v) - \sigma( A_{v}(v), u) = \sigma(\vec{A}, v) \ \sigma
(u,v).$$
 On
the other hand,  the condition $A_{v}(u) -A_{u}(v)= \sigma(v,u) \vec{A}$ for a contact prolongation from Definition
\ref{d.prolong} (3) implies
$$\sigma(A_{v}(u),v) - \sigma(A_{u}(v), v) = \sigma(v,u)\ \sigma(\vec{A},
v).$$ Combining the two equalities above, we obtain
$$\sigma(A_{v}(v), u) = \sigma(A_{u}(v), v) + 2\sigma(v, u) \
\sigma(\vec{A}, v).$$ Now, if  $v  \in \widehat{Z}$, then
$A_u \in \fg$ implies that $A_{u}(v) \in
T_{\widehat{Z},v}$. Since $T_{\widehat{Z},v}$ is isotropic with respect to $\sigma$,
this implies that $\sigma(A_{u}(v), v) = 0$. Thus the
above equality gives $$\sigma(A_{v}(v), u) = 2
\sigma(\sigma(\vec{A}, v) v, u)$$ for all $u \in V$. We
conclude that $$A_{v}(v) = 2 \sigma(\vec{A}, v)
v$$ for all nonzero $v \in \widehat{Z}$,
proving (i).

From (i), if we choose an arc $v(t) := v + t w +
t^2(\cdots)$ on $\widehat{Z}$ through $v$ in the direction of $w \in T_{\widehat{Z},v}$, we have
$$A_{v(t)}( v(t)) = 2 \sigma(\vec{A}, v(t)) \
v(t).$$ Expanding in $t$, the linear terms of both sides read
$$A_{v}(w) + A_{w}(v) = 2 \sigma(\vec{A}, v) w
+ 2 \sigma(\vec{A}, w) v.$$ From Definition \ref{d.prolong} (3),  $$A_{v}(w) - A_{w}(v) =
\sigma(v, w) \vec{A},$$ which is equal to zero,  because $v$ and $w$ are two
vectors in the  subspace $T_{\widehat{Z},v}$ isotropic with respect to $\sigma$. We obtain
$A_{v}(w) = A_{w}(v)$. It follows that
$$A_{v}(w) = \sigma(\vec{A}, v) w + \sigma(\vec{A},
w) v,$$ proving (ii). \end{proof}

Theorem \ref{t.1} is a direct consequence of Proposition \ref{p.Z_P} and the following.

\begin{proposition}\label{p.cpZ}
Let $Z \subset \BP V$ be a nondegenerate Legendrian submanifold. If $\fg^Z \subset \fcsp(V)$ has a  nonzero contact prolongation, then $Z$ is Euler-symmetric. \end{proposition}

\begin{proof}
Pick a nonzero contact prolongation $A$ of $\fg^{Z}$. First, assume
that $\vec{A} = 0$. Then we have $A_{u}(v)= A_{v}(u)$ for all
$u,v \in V$. This means   $A$ is the usual prolongation of the Lie algebra
$\fg^Z \subset \fgl(V)$,  in the
sense of the section (1.1) of \cite{HM05} or Definition 7.9 of \cite{FH18}.
Moreover, Proposition \ref{p.formula} shows $A_{v}(v)  = 0$ for all
$v \in \widehat{Z}$. This implies $A = 0$ by Lemma 2.2.1 of
\cite{HM05}, a contradiction.

Thus we assume that $\vec{A} \neq 0$.  Pick  $v \in
\widehat{Z}$ with $\omega(\vec{A}, v) =1$. The endomorphism
$A_v \in \fg^Z$ has $v$ as an eigenvector by
Proposition \ref{p.formula} (i) and acts on the tangent space of $Z$ at the point
$z \in Z$ corresponding to $v$,
$$T_{Z, z} = \Hom (\C v, T_{\widehat{Z},v}/\C
v),$$ as the identity by Proposition \ref{p.formula} (ii). In
particular, the semi-simple part of $A_v$ under
the Jordan decomposition, to be denoted by $S_v$, is a nonzero element of $\fg^Z$ and acts on $T_{Z, z}$
as the identity. Thus the $\C^{\times}$-subgroup $\{ \exp (t S_v),
t \in \C\}$ acts on $Z$ with an isolated fixed point at $z$ and
acts trivially on $\BP T_{Z,z}$. Since this is true for any $z \in Z$ outside the hyperplane $\sigma(\vec{A},
\cdot) = 0$, we see that $Z$ is Euler-symmetric. \end{proof}


\begin{thebibliography}{KSWZ}


\bibitem[Bu07]{Bu07} J. Buczy\'nski: Toric Legendrian subvarieties. Transform. Groups {\bf 12} (2007) 631-646

\bibitem[Bu08]{Bu08} J. Buczy\'nski:
Hyperplane sections of Legendrian subvarieties.
Math. Res. Lett. {\bf 15} (2008) 623-629

\bibitem[CS]{CS} A. C\v{a}p and J. Slov\'{a}k: Parabolic Geometries I. Background and general theory. Mathematical Surveys and Monographs, 154. American Mathematical Society, Providence, RI, 2009


\bibitem[FH18]{FH18} B. Fu and J.-M. Hwang: Special birational transformations of type $(2,1).$ J. Alg. Geom. {\bf 27} (2018) 55-89

\bibitem[FH17]{Euler} B. Fu and J.-M. Hwang: Euler-symmetric projective varieties.  Algebraic Geometry {\bf 7} (2020) 377-389

\bibitem[Ho]{Ho} J. Hong: Fano manifolds with geometric structures modeled after homogeneous contact manifolds. Internat. J. Math. {\bf 11} (2000) 1203-1230

\bibitem[H01]{H01} J.-M. Hwang:  Geometry of minimal rational curves on Fano manifolds.
{\it School on Vanishing Theorems and Effective
   Results in Algebraic Geometry (Trieste, 2000)},
    335-393, ICTP Lect. Notes, 6, Abdus Salam Int. Cent. Theoret. Phys., Trieste,
    2001.
%


\bibitem[HM98]{HM98} J.-M. Hwang and N. Mok: Varieties of minimal rational tangents on uniruled projective manifolds. in {\em Several complex variables}.
MSRI
Publications Vol. {\bf 37}, Cambridge University Press, Cambridge 1999, pp.351-389


\bibitem[HM05]{HM05}  J.-M.   Hwang and N. Mok:
Prolongations of infinitesimal linear automorphisms of projective
varieties and rigidity of rational homogeneous spaces of Picard
number 1 under K\"ahler deformation. Invent. Math. {\bf 160}
(2005) 591-645

\bibitem[Ke1]{Ke1} S. Kebekus: Lines on contact manifolds. J. Reine Angew. Math. {\bf 539} (2001) 167-177

\bibitem[Ke2]{Ke2} S. Kebekus: Lines on complex contact manifolds. II. Compos. Math. {\bf 141} (2005) 227-252

\bibitem[LM]{LM} J. Landsberg and L. Manivel: Legendrian varieties. Asian J. Math. {\bf 11} (2007)  341-359

\bibitem[Le]{Le}
C. LeBrun: Fano manifolds, contact structures, and quaternionic geometry. Internat. J. Math. {\bf 6} (1995) 419-437

\bibitem[Mo]{Mo} N. Mok: Recognizing certain rational homogeneous
manifolds of Picard number 1 from their varieties of minimal
rational
 tangents.  Third International Congress of Chinese Mathematicians.
 Part 1, 2,  AMS/IP Stud. Adv. Math., 42, pt.1, 2,
  Amer. Math. Soc., Providence, 2008, pp. 41-61


\bibitem[Mr]{Mr} T. Morimoto: Geometric structures on filtered manifolds. Hokkaido Math. J. {\bf 22} (1993) 263-347

\bibitem[St]{St} S. Sternberg. {\em Lectures on differential geometry}. Second Edition.
AMS Chelsea Publishing, 1983

\end{thebibliography}
\end{document}